\documentclass{amsart}
\newcommand{\mA}{\mathcal{A}}

\newcommand{\mB}{\mathcal{B}}

\usepackage{amsmath}
\usepackage{amsthm}
\usepackage{graphicx}
\usepackage{amsfonts}
\usepackage{amssymb}
\usepackage{mathrsfs}
\usepackage{amscd}
\usepackage{url}
\usepackage{color}

\newcommand{\re}{\mathbb{R}}

\newcommand{\lmd}{\lambda}

\def\rank{\operatorname{rank}}
\def\Range{\operatorname{Range}}

\newcommand{\mt}[1]{\mathtt{#1}}

\newcommand{\reff}[1]{(\ref{#1})}

\newcommand{\mc}[1]{\mathcal{#1}}

\newcommand{\supp}[1]{\mbox{supp}(#1)}

\newcommand{\bdes}{\begin{description}}
\newcommand{\edes}{\end{description}}

\newcommand{\bnum}{\begin{enumerate}}
\newcommand{\enum}{\end{enumerate}}

\newcommand{\bit}{\begin{itemize}}
\newcommand{\eit}{\end{itemize}}

\newcommand{\bea}{\begin{eqnarray}}
\newcommand{\eea}{\end{eqnarray}}

\newcommand{\be}{\begin{equation}}
\newcommand{\ee}{\end{equation}}

\newcommand{\baray}{\begin{array}}
\newcommand{\earay}{\end{array}}

\newcommand{\bca}{\begin{cases}}
\newcommand{\eca}{\end{cases}}

\newcommand{\bcen}{\begin{center}}
\newcommand{\ecen}{\end{center}}

\newcommand{\bbm}{\begin{bmatrix}}
\newcommand{\ebm}{\end{bmatrix}}

\newcommand{\bmx}{\begin{matrix}}
\newcommand{\emx}{\end{matrix}}

\newcommand{\bpm}{\begin{pmatrix}}
\newcommand{\epm}{\end{pmatrix}}

\newcommand{\btab}{\begin{tabular}}
\newcommand{\etab}{\end{tabular}}

\newcommand{\balg}{\begin{algorithm}}
\newcommand{\ealg}{\end{algorithm}}

\newcommand{\br}{\begin{remark}}
\newcommand{\er}{\end{remark}}

\newcommand{\bex}{\begin{example}}
\newcommand{\eex}{\end{example}}

\newtheorem{theorem}{Theorem}[section]
\newtheorem{pro}[theorem]{Proposition}

\newtheorem{lemma}[theorem]{Lemma}

\theoremstyle{definition}
\newtheorem{example}[theorem]{Example}

\newtheorem{alg}[theorem]{Algorithm}
\newtheorem{algorithm}[theorem]{Algorithm}
\newtheorem{remark}[theorem]{Remark}

\numberwithin{equation}{section}

\begin{document}

\title[Completely Positive Binary Tensors]
{Completely Positive Binary Tensors}

\author{Jinyan Fan}
\address{School of Mathematical Sciences, and MOE-LSC, Shanghai Jiao Tong University,
Shanghai 200240, P.R. China}
\email{jyfan@sjtu.edu.cn}

\author{Jiawang Nie}
\address{Department of Mathematics, University of California, San Diego, USA}
\email{njw@math.ucsd.edu}

\author{Anwa Zhou}
\address{Department of Mathematics, Shanghai University,
Shanghai 200444, P.R. China}
\email{zhouanwa@shu.edu.cn}


\date{}

\dedicatory{}

\begin{abstract}
A symmetric tensor is completely positive (CP) if it is a sum of
tensor powers of nonnegative vectors.
This paper characterizes completely positive binary tensors.
We show that a binary tensor is completely positive
if and only if it satisfies two linear matrix inequalities.
This result can be used to determine whether a binary tensor
is completely positive or not. When it is,
we give an algorithm for computing its cp-rank
and the decomposition. When the order is odd,
we show that the cp-rank decomposition is unique.
When the order is even, we completely characterize
when the cp-rank decomposition is unique.
We also discuss how to compute the nearest cp-approximation
when a binary tensor is not completely positive.
\end{abstract}


\maketitle

\section{Introduction}

Let $\re^n$ be the space of all real $n$-dimensional vectors.
Denote by $\re_+^n$ the nonnegative orthant, i.e.,
the set of nonnegtive vectors in $\re^n$.
Let $d$ and $n$ be positive integers.
A {\it $d$th-order $n$-dimensional tensor} $\mathcal{A}$ is
an array indexed by an integer tuple $(i_1, \ldots , i_d)$ such that
\[
\mathcal{A}=( \mA_{i_1 \ldots i_d})_{ 1 \leq i_1, \ldots , i_d  \leq n}.
\]
The tensor $\mathcal{A}$ is {\it symmetric}
if $\mathcal{A}_{i_1\ldots i_d}$ is invariant
under permutations of $(i_1, \ldots , i_d)$.
Let $\mathrm{S}^d(\mathbb{R}^{n})$ be the set of all such
$d$th-order $n$-dimensional symmetric tensors over $\re$.
The tensor $\mA$ is {\it binary} if
the dimension $n$ is two\footnote{
In most existing literature~\cite{Blekherman2015,
BCMT10,CausaR2011,ComasS2011,Comon2008,ComonO2012,Land12,Lim13,Seigal16},
binary tensors mean that the dimension $n$ equals $2$.
There also exists other work (e.g., \cite{XuLQC16}) that
uses ``binary" to mean that the tensor has $0$/$1$ entries.}.
For a vector $v \in\mathbb{R}^n$,
$v^{\otimes d}$ denotes the rank-$1$ symmetric tensor such that
for all $i_1, \ldots, i_d$
\[
(v^{\otimes d} )_{i_1,\ldots,i_d}= v_{i_1}\cdots v_{i_d}.
\]
Every symmetric tensor is a linear combination of rank-1 symmetric tensors
\cite{Comon2008}. We refer to
\cite{BerGimIda11,BCMT10,Comon2008,Nie2014b,OedingO2013}
for symmetric tensor and \cite{Land12,Lim13} for general tensors.

A symmetric tensor $\mathcal{A} \in \mathrm{S}^d(\re^{n})$ is said to be
{\it completely positive} (CP) if there exist
$v_1, \ldots, v_r \in \re^n_+$ such that
\begin{equation} \label{cdmn}
  \mathcal{A} = (v_1)^{\otimes d} + \cdots + (v_r)^{\otimes d}.
\end{equation}
The equation \eqref{cdmn} is called a {\it cp-decomposition}
of $\mathcal{A}$.
The smallest $r$ in \eqref{cdmn}
is called the {\it cp-rank} of $\mA$, denoted as $\rank_{\rm cp}(\mA)$, i.e.,
\be \label{cp-rank}
\rank_{\rm cp}(\mA)  \, = \, \min \left\{ r: \,
\mA = \sum_{i=1}^r (v_i)^{\otimes d}, v_i \in \re^n_+  \right\}.
\ee
When $r = \rank_{\rm cp}(\mA)$, \eqref{cdmn} is called a
{\it cp-rank decomposition} of $\mA$.
%
%
For nonsymmetric tensors,
a similar version of decomposition like \eqref{cdmn}
can be defined (i.e., each decomposing vector is required to be nonnegative);
such decompositions are called
{\it nonnegative tensor factorization} or
{\it nonnegative tensor decomposition}.
Generally, it is hard to check whether a tensor is completely positive or not.
The question is NP-hard even for the matrix case \cite{Dickinson11}.
To the best of the authors' knowledge,
the questions of determining cp-ranks
and computing cp-rank decompositions are mostly open.

Completely positive tensors are extensions of completely positive matrices \cite{BermanN,Burer,ZhouFan13}.
They have wide applications in
exploratory multiway data analysis and blind source separation~\cite{Cichocki},
computer vision and statistics~\cite{Shashua},
multi-hypergraphs~\cite{XuLQC16},
polynomial optimization~\cite{Burer,PVZ15}.
We refer to \cite{FanZhou15,zLuoQi15,Qi2014Nonnegative,XuLQC16}
for recent work on completely positive tensors.

Binary symmetric tensors are special cases of Hankel tensors,
which were introduced in Qi~\cite{Qi15}.
Recall that a symmetric tensor
$\mA \in \mathrm{S}^d(\mathbb{R}^{n})$
is called {\it Hankel}
if $\mathcal{A}_{i_1\ldots i_d} = \mathcal{A}_{j_1\ldots j_d}$
for all $i_1+\cdots+i_d = j_1 +\cdots +j_d$.
For binary tensors (i.e., $n=2$),
each index is either $1$ or $2$,
so $i_1+\cdots+i_d = j_1 +\cdots +j_d$ implies that
$(i_1, \ldots, i_d)$ is a permutation of $(j_1, \ldots, j_d)$.
Therefore, every binary symmetric tensor is Hankel,
and every Hankel tensor is uniquely determined by a binary symmetric tensor.
Various Hankel tensors (e.g.,
strong Hankel tensors and complete Hankel tensors)
are studied in Qi~\cite{Qi15}.
Ding, Qi and Wei~\cite{Ding16} studied inheritance properties of Hankel tensors.
%
The relations to Cauchy and Vandermonde tensors
are discussed in \cite{ChenLQ16,Xu16,XuWL16}.
%
%
The ranks and decompositions of Hankel tensors
are discussed in \cite{NieKe17}.

In this paper, we focus on binary tensors, i.e., the dimension $n=2$.
The ranks and decompositions of binary tensors
over real and complex fields are well studied in
\cite{Blekherman2015,CausaR2011,ComasS2011,ComonO2012,Sylvester}.
Nonnegative ranks and semialgebraic geometry
of nonnegative tensors are studied well in \cite{QCL16}.
The cp-ranks and cp-rank decompositions
of binary tensors are not known much in the prior existing work.
How do we efficiently check whether
a binary tensor is completely positive or not?
If it is, how do we compute its
cp-rank and cp-rank decomposition?
If it is not, how do we compute its nearest cp-approximation?
These questions are answered in this paper.
We show that a binary tensor is completely positive
if and only if it satisfies two linear matrix inequalities.
When it is completely positive,
we give algorithms for determining its cp-rank
and the cp-rank decomposition.
When the order $d$ is odd, we prove that
the cp-rank decomposition is always unique.
When the order $d$ is even, we completely characterize
when the cp-rank decomposition is unique.
%
%
When a binary tensor is not completely positive,
we show how to compute the nearest cp-approximation.

The paper is organized as follows. In Section~\ref{sc:pre},
we first review some basic results about univariate truncated moment problems,
and then characterize completely positive binary tensors.
In Section~\ref{sc:cpalg}, we give algorithms for determining
cp-ranks and cp-rank decompositions.
Section~\ref{sc:aprox} discusses how to compute
the nearest cp-approximation when a binary tensor is not completely positive.
The numerical experiments are given in Section~\ref{ne:ten:con}.
We make some conclusions in Section~\ref{sc:con}.

\section{Preliminaries}
\label{sc:pre}

A vector $y :=(y_0,y_1,\ldots,y_d) \in \re^{d+1}$ is called a
{\it truncated multi-sequence} (tms) of degree $d$.
The univariate truncated moment problem (UTMP) concerns
whether or not there exists a Borel measure $\mu$,
supported on $\re$, such that
\[
y_k  = \int_{\re} x^k d \mu \quad
(k=0,1,\ldots, d).
\]
If the above holds,
$\mu$ is called a {\it representing measure} for $y$,
and we say that $y$ admits the representing measure $\mu$.
The support of $\mu$ is denoted as $\supp{\mu}$.
The measure $\mu$ is {\it finitely atomic} if
the cardinality $|\supp{\mu}| < \infty$,
and it is {\it $r$-atomic} if $|\supp{\mu}|=r$.
If the above measure $\mu$ is $r$-atomic, then
there exist  $\lmd_1,\ldots,\lmd_r >0$
and $t_1,\ldots,t_r \in \re$ such that
 \begin{equation}\label{ratomicKmeas}
y=\lmd_1 [t_1]_d +\cdots +\lmd_r [t_r]_d,
 \end{equation}
where we denote
$
[t]_d \, := \, (1,t,\ldots,t^d).
$
The univariate truncated moment problem was well studied
in Curto and Fialkow \cite{CurFia91}.

\subsection{Some basic results for UTMP}

Suppose $d=2s$ (the even case) or $d=2s+1$ (the odd case).
For $y:=(y_0,y_1,\ldots,y_d)$, denote the Hankel matrix
\be \label{df:L1(y)}
H_1(y) := \bbm
y_0  &  y_1  &  \cdots &  y_s \\
y_1  &  y_2  &  \cdots &  y_{s+1} \\
\vdots & \vdots & \ddots & \vdots    \\
y_{s} & y_{s+1} & \cdots & y_{2s}  \\
\ebm.
\ee
If $y$ admits a representing measure on $\re$,
then $H_1(y) \succeq 0$.
(The notation $X\succeq 0$ means that
$X$ is a real symmetric positive semidefinite matrix.)
However, the converse is not necessarily true.
We use $\Range A$ to denote the range space of a matrix $A$.

\begin{theorem}(\cite[\S3]{CurFia91})\label{genefun}
Let $y := (y_0, \ldots, y_d) \in \re^{d+1}$ with $y_0 > 0$.
\bit

\item [(i)] If $d=2s+1$ is odd, then $y$ admits a representing measure on $\re$ if and only if
\be \label{cd:mea:R:odd}
H_1 ( y ) \succeq 0, \quad
(y_{s+1}, \ldots,  y_{2s+1} ) \in \, \Range\, H_1( y ).
\ee

\item [(ii)] Let $r = \operatorname{rank} H_1(y)$.
If $d=2s$ is even, then $y$ admits a representing measure on $\re$
if and only if $H_1(y)\succeq 0$ and there exist
$g_0, \ldots,  g_{r-1}$ such that

\be \label{Yg:2s-1:2s}
\bbm
y_0  &  y_1  &  \cdots &  y_{r-1} \\
y_1  &  y_2  &  \cdots &  y_{r} \\
\vdots & \vdots & \ddots & \vdots    \\
y_{2s-r} & y_{s} & \cdots & y_{2s-1}  \\
\ebm
 \bbm
g_0 \\ g_1 \\  \vdots \\  g_{r-1}
\ebm =
\bbm
y_{r} \\ y_{r+1} \\  \vdots \\ y_{2s}
\ebm.
\ee

\eit
\end{theorem}

In the next, we characterize when a tms
$y := (y_0, \ldots, y_d)$
admits a representing measure supported in $[0,1]$.
When $d=2s$ is even, denote
\be \label{df:L2(y)}
H_2(y) := \left[\baray{cccc}
y_1-y_2  &  y_2-y_3  &  \cdots &  y_{s} -y_{s+1}  \\
y_2-y_3  &  y_3-y_4  &  \cdots &  y_{s+1} -y_{s+2}   \\
\vdots & \vdots & \ddots & \vdots    \\
y_{s} -y_{s+1} & y_{s+1} -y_{s+2} & \cdots & y_{2s-1} -y_{2s}
\earay\right].
\ee
When $d=2s+1$ is odd, denote the matrices
\be \label{df:L3(y)}
H_3(y) := \left[\baray{cccc}
y_1   &  y_2  &  \cdots   & y_{s+1}\\
y_2  &  y_3  &  \cdots &  y_{s+2} \\
\vdots & \vdots & \ddots & \vdots    \\
y_{s+1} & y_{s+2} & \cdots & y_{2s+1}  \\
\earay \right],
\ee
\be \label{df:L4(y)}
H_4(y) := \left[\baray{cccc}
y_0-y_1  &  y_1-y_2  &  \cdots &  y_s-y_{s+1} \\
y_1-y_2  &  y_2-y_3  &  \cdots &  y_{s+1}-y_{s+2} \\
\vdots & \vdots & \ddots & \vdots    \\
y_{s}-y_{s+1} & y_{s+1}-y_{s+2} & \cdots & y_{2s}-y_{2s+1}  \\
\earay\right].
\ee
Note that they satisfy the relations
\[
H_2 (y)  = H_1\left( \bbm  y_1 - y_2 \\ \vdots \\ y_{2s-1} -y_{2s}  \ebm \right), \quad
H_3 (y) =  H_1\left( \bbm  y_1 \\ \vdots \\ y_{2s+1}  \ebm \right),
\]
\[
H_4 (y)  = H_1\left( \bbm  y_0 - y_1 \\ \vdots \\ y_{2s} - y_{2s+1}  \ebm \right).
\]

The following is a classical result about UTMP over $[0,1]$.

\begin{theorem}\label{flattheorem}
(\cite[\S4]{CurFia91}, \cite{KreinN1977})
Let $y = (y_0, \ldots, y_d) \in \re^{d+1}$ with $y_0>0$.
\bit

\item [(i)] If $d=2s$ is even, then
$y$ admits a representing measure supported in $[0,1]$
if and only if
$
H_1(y) \succeq 0, \, H_2(y) \succeq 0.
$

\item [(ii)] If $d=2s+1$ is odd, then
$y$ admits a representing measure supported in $[0,1]$
if and only if
$
H_3(y)  \succeq 0, \, H_4(y)  \succeq 0.
$

\eit
\end{theorem}

\subsection{UTMP and binary tensors}

For a matrix $T \in \re^{2 \times 2}$, define the linear
transformation $\mathscr{L}_{T}: \, \mathrm{S}^d( \re^2) \to \mathrm{S}^d( \re^2)$
such that, for all $\lmd_1, \ldots, \lmd_r \in \re$
and $u_1, \ldots, u_r  \in \re^2$,
\be  \label{ltrn:LT(u)}
\mathscr{L}_{T} \Big(  \sum_{i=1}^r \lambda_i (u_i)^{\otimes d}  \Big)
= \sum_{i=1}^r \lambda_i (T u_i)^{\otimes d}.
\ee
The $\mathscr{L}_{T}$ can be defined similarly
over general symmetric tensor spaces, as in \cite{Nie2014b}.
The transformation in \eqref{ltrn:LT(u)}
is in fact a multilinear matrix multiplication (cf.~\cite{Lim13}).
For all $S,T \in \re^{2 \times 2}$, we have
\[
\mathscr{L}_{S} \Big( \mathscr{L}_{T}
\Big( \sum_{i=1}^r \lambda_i u_i^{\otimes d}  \Big)  \Big) =
\sum_{i=1}^r \lambda_i (ST u_i)^{\otimes d}.
\]
Thus, when $T$ is nonsingular,
it holds that for all $\mA\in \mathrm{S}^d( \re^2)$
\[
\mathscr{L}_{T^{-1}} \Big( \mathscr{L}_{T} \Big( \mA  \Big)  \Big) = \mA.
\]

For a binary symmetric tensor $\mA \in \mathrm{S}^d( \re^2)$, let
\[
a_k = \mA_{i_1 i_2 \ldots i_d} \quad
(k = i_1 + i_2 + \cdots + i_d  - d).
\]
Clearly, $\mA$ is uniquely determined by
$a :=(a_0, a_1, \ldots, a_d)$, and vice versa.
Such a relation is denoted as
\be \label{z=B(A)}
a =  \mathfrak{B} ( \mA )\quad \mbox{or} \quad \mA=  \mathfrak{B}^{-1} (a).
\ee

\begin{lemma}\label{dec:meas}
For $\mA \in \mathrm{S}^{d} (\re^2)$ and $T = \bbm 1 & 1 \\ 0 & 1\ebm$,
let $y = \mathfrak{B} ( \mathscr{L}_T(\mA) )$.
Then, the tms $y$ of degree $d$ has the decomposition
\be \label{y=sum:[ti]d}
y = \lmd_1 [t_1]_d + \cdots + \lmd_r [t_r]_d,
\ee
if and only if $\mA$ has the decomposition
\be \label{mA=[1-ti,ti]mpowd}
\mA = \lmd_1 \bbm 1-t_1 \\ t_1 \ebm^{\otimes d} + \cdots +
\lmd_r \bbm 1-t_r \\ t_r \ebm^{\otimes d}.
\ee
\end{lemma}

\begin{proof}
``$\Rightarrow$": Suppose \reff{y=sum:[ti]d} is true. Then,
\[
\mathfrak{B}^{-1} (y)= \lmd_1 \bbm 1 \\ t_1 \ebm^{\otimes d} + \cdots +
\lmd_r \bbm 1 \\ t_r \ebm^{\otimes d},
\]
\[
\mA =\mathscr{L}_{T^{-1}} \Big( \mathfrak{B}^{-1} (y) \Big)
 = \lmd_1 \bbm 1-t_1 \\ t_1 \ebm^{\otimes d} + \cdots +
\lmd_r \bbm 1-t_r \\ t_r \ebm^{\otimes d}.
\]
So \reff{mA=[1-ti,ti]mpowd} holds.

\bigskip \noindent
``$\Leftarrow$": Suppose \reff{mA=[1-ti,ti]mpowd} is true.
Then,
\[
\mathscr{L}_T(\mA)= \lmd_1 \bbm 1 \\ t_1 \ebm^{\otimes d} + \cdots +
\lmd_r \bbm 1 \\ t_r \ebm^{\otimes d},
\]
\[
y = \mathfrak{B} ( \mathscr{L}_T(\mA) ) = \lmd_1 [t_1]_d + \cdots + \lmd_r [t_r]_d.
\]
So \reff{y=sum:[ti]d} is true.
\end{proof}

\begin{pro}\label{CPequ}
For $\mA \in \mathrm{S}^{d} (\re^2)$ and $T = \bbm 1 & 1 \\ 0 & 1\ebm$,
let $y = \mathfrak{B} ( \mathscr{L}_T(\mA) )$.
\bit

\item [(i)] If $d=2s$ is even, then $\mA$
is completely positive if and only if
\[
H_1(y) \succeq 0, \quad H_2(y) \succeq 0.
\]

\item [(ii)] If $d=2s+1$ is odd, then $\mA$
is completely positive if and only if
\[
H_3(y) \succeq 0, \quad H_4(y) \succeq 0.
\]

\eit
\end{pro}

\begin{proof}
By definition, the binary symmetric tensor $\mA$
is completely positive if and only if there exist
$a_i, b_i \geq 0, a_i + b_i >0 (i=1,\ldots,r)$ such that
\begin{equation}\label{eq:2.12}
\mA =  \bbm a_1 \\ b_1 \ebm^{\otimes d} + \cdots +
 \bbm a_r \\ b_r \ebm^{\otimes d}.
\end{equation}
Clearly,   \eqref{eq:2.12} is equivalent to
\begin{equation}\label{eq:2.13}
\mA = \lmd_1 \bbm 1-t_1 \\ t_1 \ebm^{\otimes d} + \cdots +
\lmd_r \bbm 1-t_r \\ t_r \ebm^{\otimes d},
\end{equation}
with each $\lmd_i = (a_i + b_i)^d >0$ and
$t_i = b_i/(a_i+b_i) \in [0,1]$.
Note that $y = \mathfrak{B} ( \mathscr{L}_T(\mA) )$.
By Lemma \ref{dec:meas},  \eqref{eq:2.13} is the same as
\begin{equation}\label{eq:2.14}
y  = \lmd_1 [t_1]_d + \cdots + \lmd_r [t_r]_d.
\end{equation}
Summarizing the above, we know that $\mA$
is completely positive if and only if
$y$ admits a representing measure supported in $[0,1]$.
Then, the items (i) and (ii) follow directly from Theorem \ref{flattheorem}.
\end{proof}

The cp-decomposition of a binary tensor
can be viewed as the decomposition of
a bivariate homogeneous polynomial
into a sum of nonnegative linear forms.
For $\mA \in \mathrm{S}^{d}( \re^2 )$,
define the binary $d$-form
\[
\mA(x_1, x_2) \, := \,
\sum_{i_1, \ldots, i_d =1}^2
\mA_{i_1 \ldots i_d} x_{i_1} \cdots x_{i_d}.
\]
Then $\mA$ has the decomposition \reff{eq:2.12}
if and only if
\[
\mA(x_1, x_2)
=  (a_1 x_1 +  b_1 x_2)^d + \cdots +  (a_r x_1 +  b_r x_2)^d.
\]
The complete positivity of $\mA$ requires the existence of
nonnegative $a_i, b_i$ satisfying the above.
For a binary $d$-form
$p(x_1,x_2) = \sum_{k=0}^d p_k x_1^k x_2^{d-k}$, define
\[
\mathscr{R}_{\mA}(p) :=
\sum_{k=0}^d p_k \mA_{
\underbrace{1\ldots 1}_{k   }
\underbrace{2\ldots 2}_{d-k   }
}.
\]
Then $\mathscr{R}_{\mA}(\cdot)$ is a linear functional
acting on the space of binary $d$-forms.
The decomposition \reff{eq:2.12} implies that
\[
\mathscr{R}_{\mA}(p) = p(a_1,b_1) + \cdots + p(a_r,b_r).
\]
Clearly, if $\mA$ is CP, then
$\mathscr{R}_{\mA}(p) \geq 0$
for every binary $d$-form $p$ that is copositive,
i.e., $p$ is nonnegative over
$\re_+^2$ (the nonnegative orthant).
Interestingly, the converse is also true, that is,
if $\mathscr{R}_{\mA}(p) \geq 0$
for every copositive binary $d$-form $p$,
then $\mA$ must be CP.
In fact, one can show that $\mA$ is CP if and only if
there exists a measure $\mu$ supported in $\re_+^2$ such that
\[
\mathscr{R}_{\mA}(p) = \int p \, \mt{d} \mu.
\]
We refer to \cite{FanZhou15,Nie2014,Nie2015} for the details
about tensor decompositions and truncated moment problems.

\section{cp-decompositions of binary tensors}
\label{sc:cpalg}

In this section, we give algorithms for checking
whether a binary tensor $\mA \in \mathrm{S}^{d}( \re^2 )$
is completely positive or not.
Recall that $\mA$ is CP if and only if
\[
\mA = \bbm a_1 \\ b_1 \ebm^{\otimes d}  + \cdots +
\bbm a_r \\ b_r \ebm^{\otimes d}
\]
for all $a_i, b_i \geq 0$.
The smallest possible $r$ in the above
is $\rank_{\rm cp}(\mA)$, the cp-rank of $\mA$.
When $\mA$ is not CP, we want to compute a certificate for that.
When $\mA$ is CP, we want to determine
$\rank_{\rm cp}(\mA)$ and compute a cp-rank decomposition.

We discuss cp-decomposition in two separate cases:
$d=2s+1$ is odd and $d=2s$ is even.
Note that $\mA \in \mathrm{S}^{d}( \re^2 )$
is uniquely determined by
$y :=(y_0, y_1, \ldots, y_d)$ such that
\[
y = \mathfrak{B} ( \mathscr{L}_T(\mA) ).
\]
The linear transformations $\mathscr{L}_T, \mathfrak{B}$
are as in \reff{ltrn:LT(u)} and \reff{z=B(A)},
and $T = \bbm 1 & 1 \\ 0 & 1\ebm$.
Our methods for determining cp-rank decompositions
are based on finding representing measures of the tms $y$.

\subsection{The case of odd order}

Suppose the order $d = 2s+1$ is odd.
By Proposition \ref{CPequ},
$\mA$ is completely positive if and only if
\[
H_3(y) \succeq 0, \quad H_4(y) \succeq 0.
\]
This property can be used to determine whether $\mA$ is CP or not.
If it is CP, we can further determine $\rank_{\rm cp}(\mA)$
and the unique cp-rank decomposition.

Let $ y = \mathfrak{B} ( \mathscr{L}_T(\mA) )$ be such that
$H_3(y) \succeq 0$ and $H_4(y) \succeq 0$. Then, $y_0 \geq 0$.
If $y_0 = 0$, then we must have
\[
y_0 = y_1 = \cdots = y_{2s} = y_{2s+1} = 0,
\]
hence $\mA$ is the zero tensor.
So we are mostly interested in the case $y_0 > 0$.

\begin{alg}\label{alg:odd}
For a given binary tensor $\mA \in \mathrm{S}^d(\re^2)$
of odd order $d=2s+1$, let $y := (y_0, y_1, \ldots, y_{2s+1})$
be such that $y = \mathfrak{B} ( \mathscr{L}_T(\mA) )$.

\bit

\item [Step~0] If one of $H_3(y), H_4(y)$ is not  positive semidefinite,
then $\mA$ is not  CP  and stop.
If $y_0 = 0$, then $\mA = 0$ and stop;
otherwise, go to  Step~1.

\item [Step~1] Let $ r = \rank \ H_1(y)$, and solve the linear system
\begin{equation}\label{odd:sys}
\left[\baray{lllll}
y_0  &  y_1  &  \cdots &  y_{r-1}   \\
y_1  &  y_2  &  \cdots &  y_{r}   \\
\vdots & \vdots & \ddots & \vdots    \\
y_{2s-r+1} & y_{2s-r+2} & \cdots & y_{2s}
\earay \right]
\bbm  g_0 \\ g_1 \\ \vdots \\ g_{r-1} \ebm =
\bbm   y_{r}  \\  y_{r+1}  \\ \vdots  \\ y_{2s+1}  \ebm.
\end{equation}

\item [Step~2]  Compute the roots $x_1, \ldots, x_r$ of the polynomial
\[
g(x) := g_0 + g_1 x + \cdots + g_{r-1} x^{r-1} - x^r.
\]
They must belong to the interval $[0,1]$.

\item [Step~3]  Determine the coefficients $\lmd_1, \ldots, \lmd_r$
satisfying the equation
\be \label{y=adcm:od}
y = \lmd_1 [x_1]_d + \cdots + \lmd_r [x_r]_d.
\ee
We must have $\lmd_1 >0,  \ldots,  \lmd_r >0$.

\item [Step~4] The cp-rank of $\mA$ is $r$,
and the unique cp-rank decomposition is
\be \label{cprk:dcmp:odd}
\mA  \, = \,  \bbm a_1 \\ b_1 \ebm^{\otimes d} + \cdots +
 \bbm a_r \\ b_r \ebm^{\otimes d},
\ee
where each $a_i = (\lmd_i)^{1/d}(1-x_i) \geq 0$,
$b_i = (\lmd_i)^{1/d} x_i \geq 0$.

\eit
\end{alg}

\bigskip

The properties of Algorithm \ref{alg:odd}
are summarized in the following theorem.

\begin{theorem} \label{thm:pro:odd}
Let $\mA, y$ be as in Algorithm \ref{alg:odd}.
Assume $d = 2s+1$ is odd, $y_0>0$,
$H_3(y) \succeq 0$, $H_4(y) \succeq 0$
and $r = \rank \, H_1(y)$.
Then, the following properties hold:
\bit

\item [(i)] The linear system \reff{odd:sys}
has a unique solution.

\item [(ii)] The roots $x_1, \ldots, x_{r}$
are all distinct and they belong to $[0,1]$,
and there exists a unique $(\lmd_1, \ldots, \lmd_r) > 0$
satisfying \reff{y=adcm:od}.
So \reff{cprk:dcmp:odd} must hold.

\item [(iii)] The $r$-atomic representing measure
for $y$ is unique.

\item [(iv)] The cp-rank of $\mA$ is $r$,
and the cp-rank decomposition \reff{cprk:dcmp:odd} is unique.

\eit
\end{theorem}

\begin{proof}
Since $H_3(y) \succeq 0, H_4(y) \succeq 0$,
$y$ admits a representing measure
supported in $[0,1]$, by Theorem~\ref{flattheorem},
so $\mA$ is completely positive, by Proposition~\ref{CPequ}.

(i) If $r = s+1$,
then \reff{odd:sys} clearly has a unique solution
because it is a square nonsingular linear system.
When $r \leq s$, it also has a unique solution, as shown below.
Since $y$ admits a representing measure on $\re$,
the tms $y$ is positively recursively generated
(cf.~\cite[Theorem~3.1]{CurFia91}), that is,
there exists $(g_0, \ldots, g_{r-1})$ such that
\[
g_0 y_{j-r} + \cdots + g_{r-1} y_{j-1} = y_j
\,\, (r \leq j \leq 2s+1).
\]
So \reff{odd:sys} has a solution.
Moreover, the matrix
$
H_1(y) = H_3(y) + H_4(y) \succeq 0
$
has a positive Hankel extension.
By Theorem~2.6 of \cite{CurFia91},
$\rank\, H_1(y)$ is equal to
the rank of the sequence $(y_0, y_1, \ldots, y_{2s})$,
which is the smallest number $i$ such that
\[
v_i \in \mbox{span} (v_0, \ldots, v_{i-1})
\]
for the vectors $v_j = (y_{j+\ell})_{\ell=0}^s$ ($j=0,1,\ldots, s$).
By Lemma~2.1 of \cite{CurFia91}, the solution of \reff{odd:sys}
is unique.

(ii) By Theorem~4.1(iii) of \cite{CurFia91},
$y$ admits a representing measure $\mu$ such that
$\supp{\mu} \subseteq [0,1]$ and
$\supp{\mu} = \mathcal{Z}(g(x))$ (the zero set of $g(x)$),
so \[ |\supp{\mu}| \leq r. \]
Since $r = \rank \, H_1(y)$, we must have $|\supp{\mu}| \geq r$,
whence $|\supp{\mu}| = r$.
Therefore, the roots $x_1, \ldots, x_r$ of $g$
are distinct from each other and belong to $[0,1]$.
Since $\mu$ is a representing measure for $y$
and $|\supp{\mu}| = r$,
there exist $\lmd_1, \ldots, \lmd_r > 0$
satisfying \reff{y=adcm:od}.

The equation \reff{y=adcm:od} is a Vandermonte linear system.
Since the roots $x_1, \ldots, x_r$ are all distinct,
the vector $(\lmd_1, \ldots, \lmd_r)$
satisfying \reff{y=adcm:od} must be unique.

(iii) Suppose $y$ has another $r$-atomic representing measure
\be \label{y=hat:lmd[x]:od}
y \, = \, \hat{\lmd}_1 [ \hat{x}_1 ]_d + \cdots + \hat{\lmd}_r [ \hat{x}_r ]_d,
\ee
with $\hat{x}_1, \ldots, \hat{x}_r \in \re$ being all distinct
and all $ \hat{\lmd}_1, \ldots, \hat{\lmd}_r>0$.
Let $(\hat{g}_0, \hat{g}_1, \ldots, \hat{g}_{r-1})$ be such that
\[
\hat{g}_0 + \hat{g}_1 x + \cdots + \hat{g}_{r-1} x^{r-1} - x^r   \, =  \,
-(x- \hat{x}_1) \cdots (x-\hat{x}_r).
\]
Then, the decomposition \reff{y=hat:lmd[x]:od} implies that
\[
\hat{g}_0 y_{j-r} + \hat{g}_1 y_{j-r+1} + \cdots +
\hat{g}_{r-1} x^{j-1} = y_{j} \, \, (r \leq j \leq 2s+1).
\]
So $(\hat{g}_0, \hat{g}_1, \ldots, \hat{g}_{r-1})$
is a solution to \reff{odd:sys}.
By item (i), \reff{odd:sys} has a unique solution.
Therefore,
\[
(\hat{g}_0, \hat{g}_1, \ldots, \hat{g}_{r-1}) \, = \,
(g_0, g_1, \ldots, g_{r-1}).
\]
We can see that the root set $\{ x_1, \ldots, x_r \}$
is the same as $\{ \hat{x}_1, \ldots, \hat{x}_r \}$.
Moreover, since $x_1, \ldots, x_r$ are distinct,
$(\lmd_1, \ldots, \lmd_r)$ is uniquely determined by
$(x_1, \ldots, x_r)$ and $y$.
This means that the $r$-atomic representing measure
for $y$ is unique.

(iv) By Lemma~\ref{dec:meas}, there is a one-to-one
correspondence between $r$-atomic representing measures
and cp-decompositions of length $r$.
For any representing measure $\mu$ for $y$, we must have
\[ |\supp{\mu}| \geq \rank \, H_1(y) = r, \]
so $\rank_{\rm cp} (\mA) \geq r$.
By the item (ii), we know $\rank_{\rm cp} (\mA) \leq r$,
whence $\rank_{\rm cp} (\mA) = r$.
By item (iii), the $r$-atomic representing measure for $y$ is unique,
so the cp-rank decomposition of $\mA$ is also unique.
\end{proof}

It was known that the representing measure for $y$ is unique
if $r\leq s$ and it is not unique if $r = s+1$
\cite[Theorem~3.8]{CurFia91}.
In Theorem~\ref{thm:pro:odd}(iii), we showed that
the $r$-atomic representing measure for $y$ is unique
for both $r\leq s$ and $r=s+1$.

\subsection{The case of even order}

Now we discuss the case that $d = 2s$ is even.
By Proposition \ref{CPequ},
$\mA$ is completely positive if and only if
\[
H_1(y) \succeq 0, \quad H_2(y) \succeq 0.
\]
We can use this fact to determine whether $\mA$ is CP or not.
Moreover, when $\mA$ is CP,
we can determine $\rank_{\rm cp}(\mA)$ and its cp-rank decomposition.

Let $y := (y_0, y_1, \ldots, y_{2s})$ be such that
$H_1(y) \succeq 0, H_2(y) \succeq 0$.
Then, $y_0 \geq 0$.
If $y_0=0$, then we must have
\[ y_0 = y_1 =\cdots =y_{2s} =0, \]
so $\mA$ is the zero tensor.
So we are mostly interested in the case $y_0 >0$.
The following is the algorithm for determining
$\rank_{\rm cp}(\mA)$ and its cp-rank decomposition.

\begin{alg}  \label{alg:even}
For a binary tensor $\mA \in \mathrm{S}^d(\re^2)$
of even order $d = 2s$, let $y := (y_0, y_1, \ldots, y_d)$
be such that $y = \mathfrak{B} ( \mathscr{L}_T(\mA) )$.

\bit

\item [Step~0] If one of $H_1(y), H_2(y)$ is not positive semidefinite,
then  $\mA$ is not CP and stop.
If $y_0 =0$, then $\mA$ is the zero tensor and stop;
otherwise, go to  Step~1.

\item [Step~1] Let $r = \rank \, H_1(y)$. If $r \leq s$,
solve the linear system
\be\label{even:sys1}
\bbm
y_0   &  y_1  &  \cdots   & y_{r-1}\\
y_1  &  y_3  &  \cdots &  y_{r} \\
\vdots & \vdots & \ddots & \vdots    \\
y_{2s-r} & y_{2s-r+1} & \cdots & y_{2s-1}  \\
\ebm
\bbm
g_0 \\ g_1 \\ \vdots \\ g_{r-1}
\ebm =  \bbm
 y_{r} \\  y_{r+1} \\  \vdots  \\  y_{2s}
\ebm
\ee
and go to Step~3. If $r=s+1$, go to Step~2.

\item [Step~2]  Let $z = \bbm y_{s+1} & \cdots & y_{2s} \ebm$ and
\be \label{cho:y2s+1}
y_{2s+1} := z^T \bbm
y_1   &  y_2  &  \cdots   & y_{s}\\
y_2  &  y_3  &  \cdots &  y_{s+1} \\
\vdots & \vdots & \ddots & \vdots    \\
y_{s} & y_{s+1} & \cdots & y_{2s-1}  \\
\ebm^{\dagger} z.
\ee
(The $^\dagger$ denotes the Moore-Penrose pseudo inverse.)
Then solve
\be\label{even:sys2}
\bbm
y_0   &  y_1  &  \cdots   & y_{s}\\
y_1  &  y_2  &  \cdots &  y_{s+1} \\
\vdots & \vdots & \ddots & \vdots    \\
y_{s} & y_{s+1} & \cdots & y_{2s}  \\
\ebm
\bbm
g_0 \\ g_1 \\ \vdots \\ g_{s}
\ebm =  \bbm
 y_{s+1} \\  y_{s+2} \\  \vdots  \\  y_{2s+1}
\ebm.
\ee

\item [Step~3]  Compute the roots $x_1, \ldots, x_r$ of the polynomial
\[
g(x) := g_0 + g_1 x + \cdots +  g_{r-1}  x^{r-1} - x^r.
\]
They must all belong to the interval $[0,1]$.

\item [Step~4]  Determine the coefficients $\lmd_1, \ldots, \lmd_r$
satisfying the equation
\be \label{y:lmd[x]d:ev}
y = \lmd_1 [x_1]_d + \cdots + \lmd_r [x_r]_d.
\ee
They must be all positive.

\item [Step~5] The cp-rank of $\mA$ is $r$,
and a cp-rank decomposition of $\mA$ is
\be \label{cpdcm:A:ev}
\mA =  \bbm a_1 \\ b_1 \ebm^{\otimes d} + \cdots +
 \bbm a_r \\ b_r \ebm^{\otimes d},
\ee
where each $a_i = (\lmd_i)^{1/d} (1-x_i) \geq 0$
and $b_i = (\lmd_i)^{1/d} x_i \geq 0$.

\eit

\end{alg}

For a number $t$, $y$ can be extended to the tms $\tilde{y}(t)$ of order $2s+1$:
\[
\tilde{y}(t) \, := \, (y_0, y_1, \ldots, y_{2s}, t).
\]
When $y$ admits a representing measure
supported in $[0,1]$, there exists $t \in \re$
such that this is also true for $\tilde{y}(t)$,
which is equivalent to
$ H_3( \tilde{y}(t) ) \succeq 0$, $H_4( \tilde{y}(t) ) \succeq 0$,
by Theorem~\ref{flattheorem}. Let
\[
l := \min\{t:  H_3( \tilde{y}(t) ) \succeq 0 \}, \quad
u := \max\{t:  H_4( \tilde{y}(t) ) \succeq 0 \}.
\]
By Schur's complement, the values of $l,u$ are given as
\be  \label{val:l}
l \, =  \,
\bbm y_{s+1} \\ y_{s+2} \\ \vdots \\ y_{2s}   \ebm^T
\bbm
y_1   &  y_2  &  \cdots   & y_{s}\\
y_2  &  y_3  &  \cdots &  y_{s+1} \\
\vdots & \vdots & \ddots & \vdots    \\
y_{s} & y_{s+1} & \cdots & y_{2s-1}  \\
\ebm^{\dagger}
\bbm y_{s+1} \\ y_{s+2} \\ \vdots \\ y_{2s} \ebm,
\ee
\be  \label{val:u}
u \, =  \,  y_{2s} -
\bbm  y_{s} - y_{s+1} \\  y_{s+1} - y_{s+2} \\ \vdots \\ y_{2s-1} - y_{2s}   \ebm^T
\left( H_4\left( \bbm  y_0 \\ y_1 \\ \vdots \\ y_{2s-1}  \ebm   \right) \right)^{\dagger}
\bbm  y_{s} - y_{s+1} \\  y_{s+1} - y_{s+2} \\ \vdots \\ y_{2s-1} - y_{2s}   \ebm.
\ee
Note that $l$ in \reff{val:l} is the same as
the value of $y_{2s+1}$ in \reff{cho:y2s+1}.
When $y$ admits a representing measure supported in $[0,1]$, we clearly have
\[
l\leq u.
\]
The properties of Algorithm~\ref{alg:even}
are summarized in the following theorem.

\begin{theorem}  \label{thm:pro:even}
In Algorithm \ref{alg:even}, assume that $d=2s$,
$H_1(y) \succeq 0$, $H_2(y) \succeq 0$ and $y_0 >0$.
Then, the following properties hold:

\bit

\item [(i)] When $r\leq s$, the linear system \reff{even:sys1}
has a unique solution. When $r=s+1$,
the linear system \reff{even:sys2} has a unique solution.

\item [(ii)] The roots $x_1, \ldots, x_{r}$
are all distinct and belong to $[0,1]$, and there exists
a unique $(\lmd_1, \ldots, \lmd_r) > 0$
satisfying \reff{y:lmd[x]d:ev}.
So \reff{cpdcm:A:ev} must hold.

\item [(iii)] When $r \leq s$, the representing measure for $y$
is unique (cf.~\cite[Theorem~3.10]{CurFia91}).
When $r=s+1$, the $r$-atomic representing measure for $y$, supported in $[0,1]$,
is unique if and only if $l = u$.

\item [(iv)] The cp-rank of $\mA$ is $r$.
When $r \leq s$, the cp-rank decomposition for $\mA$ is unique.
When $r=s+1$, it is unique if and only if $l = u$.

\eit

\end{theorem}
\begin{proof}
Since $H_1(y) \succeq 0, H_2(y) \succeq 0$,
we know that $y$ admits a representing measure
supported in $[0,1]$, by Theorem~\ref{flattheorem}.

(i) When $r = s+1$, \reff{even:sys2}
is a square nonsingular linear system,
so it has a unique solution.
When $r \leq s$, the matrix $H_1(y)$
has a positive Hankel extension, because
$y$ admits a representing measure on $\re$.
By Theorem~2.6 of \cite{CurFia91},
the rank of the sequence $(y_0, y_1, \ldots, y_{2s})$
is $r$, that is, $r$ is equal to the smallest number $i$ such that
\[
v_i \in \mbox{span} (v_0, \ldots, v_{i-1}),
\]
for the the vector $v_j := (y_{j+\ell})_{\ell=0}^s$, $j=0,1,\ldots, s$.
Moreover, there exist $g_0, \ldots, g_{r-1}$ such that
\[
g_0 y_{j-r} + \cdots + g_{r-1} y_{j-1} = y_j
\,\, (r \leq j \leq 2s).
\]
This shows that \reff{even:sys1} has at least one solution.
Moreover, by Lemma~2.1 of \cite{CurFia91},
the vector $(\varphi_0, \ldots, \varphi_{r-1})$
satisfying the above is unique.
So \reff{even:sys1} has a unique solution.

(ii) When $r \leq s$, choose $y_{2s+1}$ as
\be \label{new:y2s+1}
y_{2s+1} = g_0 y_{2s-r+1} + g_1 y_{2s-r} + \cdots + g_{r-1} y_{2s}.
\ee
Then, $(y_0, y_1, \ldots, y_{2s})$
and $\tilde{y} := (y_0, y_1, \ldots, y_{2s}, y_{2s+1})$
have the same unique representing measure on $\re$,
by Theorem~3.10 of \cite{CurFia91}.
So $\tilde{y}$ admits a representing measure supported in $[0,1]$.
Then, the computed values of $x_1, \ldots, x_r$
and $\lmd_1, \ldots, \lmd_r$
are the same as that obtained by applying Algorithm~\ref{alg:odd}
to the tms $\tilde{y}$ of odd order $2s+1$. By Theorem~\ref{thm:pro:odd},
$x_1, \ldots, x_r$ are distinct from each other and
belong to $[0,1]$, and there exists a unique $(\lmd_1, \ldots, \lmd_r) >0$
satisfying \reff{y:lmd[x]d:ev}.

When $r = s+1$, the value of $y_{2s+1}$ is chosen as in \reff{cho:y2s+1}.
Since $y$ admits a representing measure on $[0,1]$,
there must exist $t\in \re$ such that
this is also true for  $\tilde{y}(t)$.
By Theorem~\ref{flattheorem}, there exists $t \in \re$ such that
\[
H_3( \tilde{y}(t) ) \succeq 0,  \quad H_4( \tilde{y}(t) ) \succeq 0.
\]
So $ y_{2s+1} = l \leq u$ and
$\tilde{y} := (y_0, y_1, \ldots, y_{2s}, y_{2s+1})$
admits a representing measure supported in $[0,1]$.
The computed values of $x_1, \ldots, x_r$
and $\lmd_1, \ldots, \lmd_r$
are the same as that obtained by applying Algorithm~\ref{alg:odd}
to $\tilde{y}$. By Theorem~\ref{thm:pro:odd},
$x_1, \ldots, x_r$ are all distinct and belong to $[0,1]$,
and there exists a unique $(\lmd_1, \ldots, \lmd_r) >0$
satisfying \reff{y:lmd[x]d:ev}.

Since \reff{y:lmd[x]d:ev} is equivalent to \reff{cpdcm:A:ev},
\reff{cpdcm:A:ev} must be true.

(iii) When $r \leq s$, the representing measure of $y$ is unique
as shown in \cite[Theorem~3.10]{CurFia91}).
When $r =s+1$, the uniqueness depends on whether $l=u$ or not. Note that
\reff{y:lmd[x]d:ev} is true if and only if
there exists $t$ such that
\be \label{dcmp:tldy(t)}
\tilde{y}(t) = \lmd_1 [x_1]_{d+1} + \cdots + \lmd_r [x_r]_{d+1}.
\ee
When $l = u$, $l$ is the unique value for $t$
such that $\tilde{y} (t)$
admits a representing measure supported in $[0,1]$.
By Theorem~\ref{thm:pro:odd}, $\tilde{y}(l)$, and hence $y$,
has a unique $r$-atomic representing measure.
When $l < u$, there are infinitely many values for $t$
such that $\tilde{y}(t)$ admits a representing measure supported in $[0,1]$.
For distinct values of $t$, the decomposition \reff{dcmp:tldy(t)}
is different. So, when $y$ has a unique $r$-atomic representing measure
supported in $[0,1]$, we must have $l=u$.

(iv)
By Lemma~\ref{dec:meas}, there is a one-to-one
correspondence between cp-decompositions of length $r$ for $\mA$
and $r$-atomic representing measures for $y$, supported in $[0,1]$.
For any representing measure $\mu$ for $y$,
we always have \[ |\supp{\mu}| \geq \rank\, H_1(y)=r, \]
so $\rank_{\rm cp}(\mA) \geq r$.
By the item (ii), we know $\mA$ has a cp-decomposition of length $r$,
so $\rank_{\rm cp}(\mA) \leq r$.
Therefore, $\rank_{\rm cp}(\mA) = r$.
The statement about uniqueness of cp-rank decompositions
follows directly from item (iii).
\end{proof}

It was known that the representing measure for $y$ is unique
if $r\leq s$ (\cite[Remark~4.5]{CurFia91}).
When $r=s+1$, not much is known about the uniqueness
in the prior existing work on truncated moment problems.
In Theorem~\ref{thm:pro:even}(iii), we completely characterized when
the $r$-atomic representing measure for $y$,
supported in $[0,1]$, is unique for the case $r=s+1$.

\section{The nearest cp-approximation}
\label{sc:aprox}

In applications, a binary tensor $\mA \in \mathrm{S}^{d}(\re^2)$
may not be completely positive, or it is perturbed
from a completely positive one.
So people are often interested in computing
a completely positive tensor $\mB$ that is closest to $\mA$.
Such a tensor $\mB$ can be found by solving a convex optimization problem.

For a binary tensor $\mc{X} \in \mathrm{S}^{d}(\re^2)$,
define its Hilbert-Schmidt norm as
\be \label{norm:X}
\| \mc{X} \| \, := \,
\Big( \sum_{i_1, \ldots, i_d = 1}^2  | \mc{X}_{i_1 \ldots i_d} |^2 \Big)^{1/2}.
\ee
This norm gives a metric of distance
in the space $\mathrm{S}^{d}(\re^2)$.
Let $z = \mathfrak{B}( \mathscr{L}_T(\mc{X}) )$.
By Proposition~\ref{CPequ}, when $d=2s$ is even,
$\mc{X}$ is completely positive if and only if
\[
H_1(z) \succeq 0, \quad H_2(z) \succeq 0.
\]
When $d=2s+1$ is odd,
$\mc{X}$ is completely positive if and only if
\[
H_3(z) \succeq 0, \quad H_4(z) \succeq 0.
\]
Using the above characterizations, we can compute the CP tensor $\mB$
that is closest to $\mA$.
When $d=2s$ is even, $\mB$ is the optimizer of
\be \label{opt:near:ev}
\left\{ \baray{rl}
\min &  \| \mc{X} - \mA \| \\
s.t. & z  = \mathfrak{B}( \mathscr{L}_T(\mc{X}) ), \\
& H_1(z) \succeq 0, \, H_2(z) \succeq 0.
\earay \right.
\ee
When $d=2s+1$ is odd, $\mB$ is the optimizer of
\be \label{opt:near:odd}
\left\{ \baray{rl}
\min &  \| \mc{X} - \mA \| \\
s.t. & z  = \mathfrak{B}( \mathscr{L}_T(\mc{X}) ), \\
& H_3(z) \succeq 0, \, H_4(z) \succeq 0.
\earay \right.
\ee
Both \reff{opt:near:ev} and \reff{opt:near:odd}
can be solved as a semidefinite program (SDP).
%
%

\section{Numerical Experiments}
\label{ne:ten:con}

In this section, we give numerical experiments
for checking whether a binary tensor is completely positive or not.
If it is, we compute its cp-rank and the cp-rank decomposition.
If  it is not, we compute its nearest cp-approximation by solving \reff{opt:near:ev} or \reff{opt:near:odd} with the SDP solver SeDuMi \cite{Sturm}.
The computation is implemented in MATLAB R2015b,
on a Lenovo Laptop with CPU@2.50GHz and RAM 4.00 GB.
For convenience of presentation, four decimal digits are displayed
for showing computational results.
For a given binary tensor $\mA \in \mathrm{S}^{d}(\re^2)$,
we display $\mA$ by showing the vector $a$ such that
\[
a_k \, = \, \mA_{i_1 \ldots i_d}  \quad \mbox{ if } \quad
k = i_1 + \cdots + i_d - d.
\]
That is, $a = \mathfrak{B}(\mA)$
(see \reff{z=B(A)} for the definition of $\mathfrak{B}$).

\bex\label{Exarand4} \upshape
Consider the binary tensor $\mA \in \mathrm{S}^{4}(\re^2)$ such that
\[
a  \, = \, [   13    ,\,     5    ,\,     2     ,\,    1    ,\,     1].
\]
The tms $y =[ 50   ,\,   15   ,\,    5   ,\,    2   ,\,    1]$.
The order $d=4$ is even.
We apply Algorithm \ref{alg:even} and obtain
\[
H_1(y)=\bbm  50& 15& 5\\
             15& 5 & 2 \\
             5 &  2 & 1 \ebm
, \quad
H_2(y)=\bbm  10 & 3 \\
             3 & 1 \ebm.
\]
They are both positive semidefinite,
so $\mA$ is completely positive.
The rank of $H_1(y)$ is $2$,
so the cp-rank is also $2$.
By Theorem~\ref{thm:pro:even},
the cp-rank decomposition is unique, which is given as
\[
\mA =
\bbm  0.3523 \\  0.9223 \ebm^{\otimes 4} +
\bbm  1.8983 \\  0.7251 \ebm^{\otimes 4}.
\]

\eex

\bex\label{Exarand5} \upshape
Consider the binary tensor $\mA \in \mathrm{S}^{d}(\re^2)$  such that
\[
a=[1, \, 1/2, \,1/3, \,  \ldots, 1/(d+1)].
\]
We determine its cp-rank and the cp-decomposition
for $d=4,5,6,7$.

\smallskip
\noindent
{\bf Case $d=4$:}
We apply Algorithm \ref{alg:even} and get the cp-decomposition
\[
\mA =
\bbm  0.7833 \\  0.6618 \ebm^{\otimes 4} +
\bbm  0.8461 \\  0.3004 \ebm^{\otimes 4} +
\bbm  0.5774 \\  0.0000 \ebm^{\otimes 4}.
\]
The cp-rank $r$ of $\mA$ is $3$ and $s=2$.
Since $r=s+1$, $l=9/100 < u= 71/780$,
we know the cp-rank decomposition is not unique, by Theorem \ref{thm:pro:even}.

\smallskip \noindent
{\bf Case $d=5$:}
We apply Algorithm \ref{alg:odd} and get the cp-decomposition
\[
\mA =
\bbm  0.7740 \\  0.6868 \ebm^{\otimes 5} +
\bbm  0.8503 \\  0.4251 \ebm^{\otimes 5} +
\bbm  0.7740 \\  0.0872 \ebm^{\otimes 5}.
\]
The cp-rank $r$ of $\mA$ is $3$.
The cp-rank decomposition is unique, by Theorem \ref{thm:pro:odd}.

\smallskip \noindent
{\bf Case $d=6$:}
We apply Algorithm \ref{alg:even} and get the cp-decomposition
\[
\mA =
\bbm  0.7772 \\  0.7084 \ebm^{\otimes 6} +
\bbm  0.8541 \\  0.5044 \ebm^{\otimes 6} +
\bbm  0.8308 \\  0.1764 \ebm^{\otimes 6} +
\bbm  0.6300 \\  0.0000 \ebm^{\otimes 6}.
\]
The cp-rank $r$ of $\mA$ is $4$ and $s=3$.
Since $r=s+1$, $l=899/13545 < u= 251/3780$,
we know the cp-rank decomposition is not unique, by Theorem \ref{thm:pro:even}.

\smallskip \noindent
{\bf Case $d=7$:}
We apply Algorithm \ref{alg:odd} and get the cp-decomposition
\[
\mA =
\bbm  0.7789 \\  0.7248 \ebm^{\otimes 7} +
\bbm  0.8521 \\  0.5709 \ebm^{\otimes 7} +
\bbm  0.8521 \\  0.2812 \ebm^{\otimes 7} +
\bbm  0.7789 \\  0.0541 \ebm^{\otimes 7}.
\]
The cp-rank $r$ of $\mA$ is $4$.
The cp-rank decomposition is unique, by Theorem \ref{thm:pro:odd}.

\eex

\bex\label{UnqCP:d=6} \upshape
Consider the binary tensor $\mA \in \mathrm{S}^{d}(\re^2)$  such that
\be \label{A:cpdm:r=5}
\mA \, = \,
\bbm  1 \\  0 \ebm^{\otimes d} +
\bbm  1 \\  2 \ebm^{\otimes d} +
\bbm  1 \\  1 \ebm^{\otimes d} +
\bbm  2 \\  1 \ebm^{\otimes d} +
\bbm  0 \\  1 \ebm^{\otimes d}.
\ee
Clearly, it is a CP tensor.
We investigate the uniqueness of its cp-rank decompositions
for different values of the order $d$.

\smallskip
\noindent
\textbf{Case $d=6$:} The vectors $a,y$ are
\[
a=[67, 35, 21, 17, 21, 35, 67],
\]
\[
y=[1524, 762, 422, 252, 158, 102, 67].
\]
We have that $s=3, r=4$ and
\[
l  =  11213/252  \, < \,
u  =  11215/252.
\]
By Theorem \ref{thm:pro:even},
the cp-rank is $4$ and
the cp-rank decomposition is not unique.
The computed cp-decomposition is
\be \label{A:cpdm:dec6}
\mA \, = \,
\bbm  0.2336 \\  1.2470 \ebm^{\otimes 6} +
\bbm  1.0295 \\  1.9903 \ebm^{\otimes 6} +
\bbm  2.0032 \\  1.0136 \ebm^{\otimes 6} +
\bbm  1.0296 \\  0.0000 \ebm^{\otimes 6}.
\ee

\smallskip
\noindent
\textbf{Case $d=7$:} The vectors $a,y$ are:
\[
a=[131, 67, 37, 25,  25, 37,  67, 131],
\]
\[
y=[4504, 2252, 1248, 746, 468, 302, 198, 131].
\]
We have that $s=3, r=4$.
By Theorem \ref{thm:pro:odd},
the cp-rank is $4$ and
the cp-rank decomposition is unique.
The computed cp-decomposition is
\be \label{A:cpdm:dec7}
\mA \, = \,
\bbm  0.1966 \\  1.1843 \ebm^{\otimes 7} +
\bbm  1.0138 \\  1.9969 \ebm^{\otimes 7} +
\bbm  1.9969 \\  1.0138 \ebm^{\otimes 7} +
\bbm  1.1843 \\  0.1966 \ebm^{\otimes 7}.
\ee

\smallskip
\noindent
\textbf{Case $d=8$:} The vectors $a,y$ are
\[
a=[259, 131, 69, 41, 33, 41, 69, 131, 259],
\]
\[
y=[13380, 6690, 3710, 2220, 1394, 900, 590, 390, 259].
\]
We have that $s=4, r=5$ and
\[
l \, = \, u  \, = \, 345/2.
\]
By Theorem \ref{thm:pro:even},
the cp-rank is $5$ and
the cp-rank decomposition is unique.
The computed cp-decomposition is the same as \reff{A:cpdm:r=5}.

\smallskip
\noindent
\textbf{Case $d=9$:} The vectors $a,y$ are
\[
a=[515, 259, 133, 73, 49, 49, 73, 133, 259, 515],
\]
\[
y=[39880, 19940, 11064, 6626, 4164, 2690, 1764, 1166, 774, 515].
\]
We have that $s=4, r=5$.
By Theorem \ref{thm:pro:odd},
the cp-rank is $5$ and
the cp-rank decomposition is unique.
The computed cp-decomposition is the same as \reff{A:cpdm:r=5}.

\eex

\bex  \upshape  \label{exm:exp:dis}
Consider $\mA \in \mathrm{S}^{d}(\re^2)$ such that
$a \,= \, \mathfrak{B} (\mA) $ is given as
\[
a_k = \int_0^{\infty} x^k  e^{-x} \mt{d} x.
\]
The $a_k$'s are the moments of
the standard exponential distribution.

\medskip \noindent
\textbf{Case $d=9$:} The vector $a$ is
\[
a=[1, 1, 2, 6, 24, 120, 720,  5040, 40320,  362880].
\]
Since $d=9$ is odd, we apply Algorithm~\ref{alg:odd}.
Since $H_3(y)$ and $H_4(y)$ are both positive semidefinite, $\mA$ is CP.
We have that $s=4, r=5$.
By Theorem \ref{thm:pro:odd},
the cp-rank is $5$, and
the cp-rank decomposition is unique and it is
\[
\mA \, = \,
\bbm  0.3058 \\  3.8653 \ebm^{\otimes 9} +
\bbm  0.5353 \\  3.7934 \ebm^{\otimes 9} +
\bbm  0.7509 \\  2.7007 \ebm^{\otimes 9} +
\bbm  0.9029 \\  1.2761 \ebm^{\otimes 9} +
\bbm  0.9303 \\  0.2452 \ebm^{\otimes 9}.
\]

\medskip \noindent
\textbf{Case $d=10$:} The vector $a$ is
\[
a=[1, 1, 2, 6, 24, 120, 720,  5040, 40320,  362880, 3628800].
\]
The order is even, so we apply Algorithm~\ref{alg:even}.
Since $H_1(y)$ and $H_2(y)$ are both positive semidefinite, $\mA$ is CP.
We have that $s=5, r=6$ and
\[
l  =  13375670400/4051  \, < \,
u  =  2552306400/773.
\]
By Theorem \ref{thm:pro:even},
the cp-rank is $6$ and
the cp-rank decomposition is not unique.
The cp-decomposition obtained by Algorithm~\ref{alg:even} is
\[
\mA \, = \,
\bbm  0.2941 \\  4.1935 \ebm^{\otimes 6} +
\bbm  0.5031 \\  4.2253 \ebm^{\otimes 6} +
\bbm  0.7055 \\  3.2531 \ebm^{\otimes 6} +
\bbm  0.8662 \\  1.8302 \ebm^{\otimes 6} +
\bbm  0.9443 \\  0.5827 \ebm^{\otimes 6} +
\bbm  0.8360 \\  0.0000 \ebm^{\otimes 6}.
\]
\eex

\bex\label{Exarand3} \upshape
Consider the binary tensor $\mA \in \mathrm{S}^{7}(\re^2)$ such that
\[
a \,= \, [1  ,\, 1 ,\, 1 ,\, 1 ,\, 1 ,\, 2 ,\, 1 ,\, 2].
\]
The order $d=7$ is odd, so we apply Algorithm \ref{alg:odd}.
Since $H_3(y)$ is not positive semidefinite, $\mA$ is not CP.
To get the nearest CP tensor,
we solve the optimization problem~\eqref{opt:near:odd}.
The nearest CP tensor ${\mc{X}}^* \in \mathrm{S}^{7}(\re^2)$ is
\[
{\mc{X}}^* =
\bbm  0.8362 \\  0.3398 \ebm^{\otimes 7} +
\bbm  0.8426 \\  1.0005 \ebm^{\otimes 7} +
\bbm  0.8359 \\  0.3396 \ebm^{\otimes 7} +
\bbm  0.8416 \\  0.9993 \ebm^{\otimes 7}.
\]
The distance $\| \mA - \mc{X}^* \| \approx 3.4623$.
\eex

\bex\label{Exarand1} \upshape
Consider the binary tensor $\mA \in \mathrm{S}^{8}(\re^2)$ such that
\[
a=[1 ,\, 2 ,\, 1 ,\, 2 ,\, 1 ,\, 2 ,\, 1 ,\, 2 ,\, 1].
\]
Since the order $d=8$ is even, we apply Algorithm \ref{alg:even}.
Neither $H_1(y)$ nor $H_2(y)$ is positive semidefinite,
so $\mA$ is not completely positive.
We compute the nearest CP tensor by solving
the optimization problem~\eqref{opt:near:ev}.
The nearest CP tensor
\[
{\mc{X}}^* \, = \,   \bbm 1.0520 \\ 1.0520 \ebm^{\otimes 8}.
\]
The distance $\| \mA - \mc{X}^* \| \approx 8.0000$.
\eex

\bex  \upshape  \label{exm:a:Gauss}
Consider $\mA \in \mathrm{S}^{d}(\re^2)$ such that
$a = \mathfrak{B} (\mA)$ is given as
\[
a_k = \int x^k \cdot \frac{1}{\sqrt{2\pi}} e^{ -\frac{x^2}{2} }   \mt{d} x.
\]
The $a_k$'s are the moments of the Gaussian distribution
with mean $0$ and variance $1$.

\medskip \noindent
\textbf{Case $d=6$:}
$
a=[1, 0, 1, 0, 3, 0, 15].
$
The order $d=6$ is even, so we apply Algorithm \ref{alg:even}.
Since $H_2(y)$ is not positive semidefinite, $\mA$ is not CP.
By solving the optimization problem~\eqref{opt:near:ev},
we get the nearest CP tensor
\[
{\mc{X}}^* =
\bbm  0.0000 \\  1.5293 \ebm^{\otimes 6} +
\bbm  0.7703 \\  1.0239 \ebm^{\otimes 6} +
\bbm  0.7591 \\  1.0090 \ebm^{\otimes 6} +
\bbm  0.9183 \\  0.0000 \ebm^{\otimes 6}.
\]
%
%
The distance $\| \mA - \mc{X}^* \| \approx 9.1199$.

\smallskip \noindent
\textbf{Case $d=7$:}
$
a=[1, 0, 1, 0, 3, 0, 15, 0].
$
The order $d=7$ is odd, so we apply Algorithm \ref{alg:odd}.
Since $H_3(y)$ is not positive semidefinite, $\mA$ is not CP.
By solving  ~\eqref{opt:near:odd},
we get the nearest CP tensor
\[
{\mc{X}}^* =
\bbm   0.6369  \\ 1.2438 \ebm^{\otimes 7} +
\bbm   0.6341 \\  1.2382 \ebm^{\otimes 7} +
\bbm   0.6322  \\ 1.2347 \ebm^{\otimes 7} +
\bbm   0.9812 \\ 0.0000 \ebm^{\otimes 7}.
\]
%
%
The distance
$
\| \mA - \mc{X}^* \| \approx31.4464.
$

\eex

\section{Conclusion}
\label{sc:con}

This paper gave a characterization
for completely positive binary tensors.
We showed that a binary tensor is CP
if and only if it satisfies two linear matrix inequalities.
Based on this, we proposed an algorithm for
checking complete positivity.
When a binary tensor is completely positive,
the algorithm can compute its cp-rank
and cp-decomposition;
when it is not, the algorithm can tell
that it is not CP.
For the odd order case,
we showed that the cp-rank decomposition is unique.
For the even order case,
the cp-rank decomposition may or may not be unique;
we characterized when the cp-rank decomposition is unique.
Moreover, we showed how to compute the nearest cp-approximation
when a binary tensor is not completely positive.

One would naturally wonder whether or not our results
can be extended to the case of
higher dimensional tensors.
Indeed, Fan and Zhou \cite{FanZhou15}
already worked on general CP tensors.
They proposed a semidefinite relaxation algorithm
for higher dimensional case.
However, when the dimension is bigger than $2$,
characterizing complete positivity is much harder
than the binary case.
This is because, for binary tensors,
checking their complete positivity is equivalent
to a univariate truncated moment problem,
which has a clean solution, see Section~2.1.
For higher dimensional tensors,
checking their complete positivity is equivalent
to a multivariate truncated moment problem,
which is unfortunately much harder
and does not have a clean solution.
We refer to \cite{FanZhou15,Nie2014,Nie2015} for details.
An important future work is to design efficient
algorithms for computing cp-ranks and cp-rank decompositions
of higher dimensional CP tensors.

\bigskip \noindent
{\bf Acknowledgement} \,
Jinyan Fan was partially supported by the NSFC grant 11571234.
Jiawang Nie was partially supported by the NSF grants DMS-1417985 and DMS-1619973.
Anwa Zhou was partially supported by the NSFC grant 11701356, the National Postdoctoral Program for Innovative Talents grant BX201600097
and Project Funded by China Postdoctoral Science Foundation grant 2016M601562.


\begin{thebibliography}{10}


%
%





\bibitem{BermanN} {\sc A.~Berman and N.~Shaked-Monderer},
{\em Completely Positive Matrices}, World
Scientific, 2003.



\bibitem{BerGimIda11}
{\sc A.~Bernardi, A.~Gimigliano and M.~Id\`{a}},
{\em Computing symmetric rank for symmetric tensors},
Journal of Symbolic Computation, 46 (2011), pp. 34--53.


%



\bibitem{Blekherman2015} {\sc G. Blekherman},
{\em Typical real ranks of binary forms},
Foundations of Computational Mathematics,
15 (2015), pp. 793--798.







\bibitem{BCMT10}
{\sc J.~Brachat, P.~Comon, B.~Mourrain and E.~Tsigaridas},
{\em Symmetric tensor decomposition},
Linear Algebra and its Applications, 433 (2010),
pp. 1851--1872.




\bibitem{Burer} {\sc S.~Burer},
{\em On the copositive representation of binary and continuous nonconvex quadratic
programs}, Mathematical Programming, Ser. A, 120 (2009), pp.~479--495.


\bibitem{CausaR2011} {\sc A. Causa and R. Re},
{\em On the maximum rank of a real binary form}, Annali di Matematica, 190 (2011), pp. 55--59.





\bibitem{Cichocki} {\sc A. Cichocki, R. Zdunek, A. H. Phan and S. Amari}, {\em Nonnegative Matrix and
Tensor Factorizations: Applications to Exploratory Multiway Data Analysis and
Blind Source Separation}, Wiley, New York, 2009.

\bibitem{ChenLQ16} {\sc H. Chen, G. Li and L. Qi}, {\em Further results on Cauchy tensors and Hankel
tensors}, Applied Mathematics and Computation, 275 (2016), pp. 50--62.


%
%


\bibitem{ComasS2011} {\sc  G.~Comas and M.~Seiguer}, {\em On the rank of a binary form},
Foundations of Computational Mathematics, 11 (2011), pp. 65--78.

%
%


\bibitem{Comon2008} {\sc P. Comon, G. Golub, L. H. Lim and B. Mourrain},
{\em Symmetric tensors and symmetric tensor rank},
 SIAM Journal on Matrix Analysis and Applications, 30 (2008), pp. 1254--1279.



\bibitem{ComonO2012} {\sc P. Comon and G. Ottaviani},
{\em On the typical rank of real binary forms},
 Linear and Multilinear Algebra, 60 (2012), pp. 657--667.


%
%




\bibitem{CurFia91} {\sc R.~Curto and L.~Fialkow},
{\em Recursiveness, positivity, and truncated moment problems},
Houston Journal of Mathematics, 17 (1991), pp. 603--635.





\bibitem{Dickinson11} {\sc P. J. Dickinson and L. Gijben},
{\em On the computational complexity of membership problems for the completely positive
cone and its dual}, Computational Optimization and Applications,
57 (2014), pp. 403--415.


%
%

\bibitem{Ding16}
{\sc W. Ding, L. Qi and Y. Wei},
{\em Inheritance properties and sum-of-squares decomposition of Hankel tensors:
Theory and algorithm},
BIT Numerical Mathematics, 57 (2017),  pp. 169-190.



\bibitem{FanZhou15} {\sc J. Fan and A. Zhou},
{\em A semidefinite algorithm for completely positive tensor decomposition},
Computational Optimization and Applications, 66 (2017), pp. 267--283.



%
%
%



%












%
%

%
%


\bibitem{KreinN1977} {\sc M. G. Krein and A. A. Nudelman},
{\em  The Markov moment problem and extremal problems},
Translations of Mathematical Monographs,
Vol. 50. American Mathematical Society, Providence, R.I., 1977.



\bibitem{Land12}
{\sc J.M.~Landsberg}, {\em Tensors: geometry and applications}. Graduate
Studies in Mathematics, 128. American Mathematical Society,
Providence, RI, 2012.


%
%

%
%

%
%
%








\bibitem{Lim13}
{\sc L.-H. Lim},
{\em Tensors and hypermatrices},
in: L. Hogben (Ed.),
Handbook of Linear Algebra, 2nd Ed., CRC Press, Boca Raton, FL, 2013.


\bibitem{zLuoQi15}
{\sc Z.~Luo and L.~Qi},
{\em Completely positive tensors: Properties, easily checkable
subclasses and tractable relaxations}, SIAM Journal on Matrix Analysis and
Applications, 37 (2016), pp. 1675--1698.




\bibitem{Nie2014}
{\sc J. Nie},
{\em The $\mathcal{A}$-truncated K-moment problem},
Foundations of Computational Mathematics,
14 (2014), pp. 1243--1276.


\bibitem{Nie2015}
{\sc J. Nie},
{\em Linear optimization with cones of moments and nonnegative polynomials},
Mathematical Programming,
Ser. B, 153 (2015), pp. 247--274.

%

%
%


\bibitem{Nie2014b}
{\sc J. Nie},
{\em Generating polynomials and symmetric tensor decompositions},
Foundations of Computational Mathematics, 17 (2017), pp. 423--465.


\bibitem{NieKe17}
{\sc J. Nie and K. Ye},
{\em Hankel tensor decompositions and ranks},
Preprint 2017.
\url{arXiv:1706.03631 [math.AG]}





 \bibitem{OedingO2013}
 {\sc  L. Oeding, G. Ottaviani},
 {\em Eigenvectors of tensors and algorithms for Waring decomposition},
 Journal of Symbolic Computation, 54 (2013), pp. 9--35.

%


\bibitem{PVZ15}
{\sc J. Pe\~{n}a, J. Vera and L. Zuluaga},
{\em Completely positive reformulations for polynomial optimization},
Mathematical Programming, Ser. B,
151 (2015), pp. 405--431.

%

\bibitem{Qi2014Nonnegative}
{\sc L. Qi, C. Xu and Y. Xu},
{\em Nonnegative tensor factorization, completely positive
tensors and an hierarchical elimination algorithm},
SIAM Jounral on Matrix Analysis and Applications, 35 (2014), pp. 1227--1241.



\bibitem{Qi15} {\sc L. Qi},
{\em Hankel tensors: Associated Hankel matrices and Vandermonde decomposition},
Communications in Mathematical Sciences, 13 (2015), pp. 113--125.


\bibitem{QCL16}
{\sc Y. Qi, P. Comon, and L.-H. Lim},
{\em Semialgebraic geometry of nonnegative tensor rank},
SIAM Journal on Matrix Analysis and Applications, 37 (2016), pp. 1556--1580.


%

%
%

\bibitem{Seigal16}
{\sc A. Seigal},
{\em Gram determinants of real binary tensors},
Preprint, 2016.
\url{arXiv:1612.04420 [math.SP]}



\bibitem{Shashua}
{\sc A. Shashua and T. Hazan},
{\em Non-negative tensor factorization with applications
to statistics and computer vision},
ACM International Conference Proceeding
Series: Proceedings of the 22nd international conference on Machine learning,
119 (2005), pp. 792--799.



\bibitem{Sturm} {\sc J.~F. Sturm},
{\em SeDuMi 1.02: A MATLAB toolbox for optimization over symmetric cones}, Optimization
Methods and Software, 11 \& 12 (1999), pp.~625--653.

\bibitem{Sylvester} {\sc J.J. Sylvester}, {\em Sur une extension d$\acute{u}$n th$\acute{e}$or$\grave{e}$me de Clebsch relatif aux courbes du quatri$\grave{e}$me degr$\acute{e}$}, C. R. Math. Acad. Sci. Paris, 102 (1886), pp. 1532--1534.



%
%


\bibitem{Xu16}
{\sc  C. Xu}, {\em Hankel tensors, Vandermonde tensors and their positivities},
Linear Algebra and its Applications,
491 (2016), pp. 56--72.

\bibitem{XuWL16}
{\sc C. Xu, M. Wang and X. Li}, {\em Generalized Vandermonde tensors}, Frontiers of
Mathematics in China, 11 (2016), pp. 593--603.

\bibitem{XuLQC16}
{\sc C. Xu, Z. Luo, L. Qi and Z. Chen},
{\em \{0,1\} completely positive tensors and multi-hypergraphs},
Linear Algebra and its Applications,
510 (2016), pp. 110--123.




\bibitem{ZhouFan13}  {\sc A. Zhou  and J. Fan},
 {\em The CP-matrix completion problem},
 SIAM Jounral on Matrix Analysis and Applications, 35 (2014), pp. 127--142.





\end{thebibliography}

\end{document}